\documentclass[11pt]{article}

\usepackage[left=2cm,right=2cm,top=3cm,bottom=3cm]{geometry}

\usepackage{amsmath, amssymb, latexsym, amsthm}
\usepackage{fullpage}
\usepackage{pstricks,pst-node,pst-pdf}
\usepackage{color}
\usepackage{tikz}
\usepackage{graphicx}
\usepackage{stmaryrd,amsbsy,enumerate}
\usepackage{hyperref}
\usepackage{mathscinet}
\usetikzlibrary{calc}
\usepackage{enumitem}

\def\COMMENT#1{}
\let\COMMENT=\footnote 

\newtheorem{theorem}{Theorem} 
\newtheorem{theorem*}{Theorem} 
\newtheorem{proposition}[theorem]{Proposition} 

\newtheorem{corollary}[theorem]{Corollary}
\newtheorem{lemma}[theorem]{Lemma}
\newtheorem{conjecture}[theorem]{Conjecture}

\newtheorem{claim}[theorem]{Claim}

\theoremstyle{definition}

\newtheorem{construction}[theorem]{Construction}

\theoremstyle{remark}

{\end{enumerate}}

\newcommand{\oururl}{\url{http://www.math.uiuc.edu/~jobal/cikk/rbt}}

\newcommand{\cF}{\mathcal{F}}
\newcommand{\cA}{\mathcal{A}}
\newcommand{\cK}{\mathcal{K}}
\newcommand{\eps}{\varepsilon}
\renewcommand{\epsilon}{\varepsilon}
\newcommand{\Nat}{\mathbb{N}}
\newcommand{\Rat}{\mathbb{Q}}
\newcommand{\Real}{\mathbb{R}}

\def\RBT{\ensuremath{\mathrm{RBT}}}
\def\TCT{\ensuremath{\mathrm{TCT}}}
\def\RB{\ensuremath{\mathrm{RB}}}
\def\MONOT{\ensuremath{\mathrm{MONOT}}}

\DeclareMathOperator{\im}{im}
\DeclareMathOperator{\Hom}{Hom}

\newcounter{lth}
\newcommand{\JV}{ }

\newcommand{\flagsboundA}{{14659368409762259334120822071345940493779\over 5575186299632655785383929568162090376495104}}
\newcommand{\flagsboundB}{{11151645199111581268390153119301740786646069 \over 27875931498163278926919647840810451882475520}}
\newcommand{\flagsboundC}{{265485807942351943716784898403205143897069\over 2787593149816327892691964784081045188247552}}
\newcommand{\flagsboundD}{{5576885389284149539505627500589996258413877 \over 16725558898897967356151788704486271129485312}}

\begin{document}

\title{Rainbow triangles in three-colored graphs}
\author{
J\'{o}zsef Balogh\thanks{ Department of Mathematics, University of Illinois, Urbana, IL 61801, USA and Bolyai Institute, University of Szeged, Szeged, Hungary E-mail: {\tt jobal@math.uiuc.edu}.
    Research is partially supported by Simons Fellowship, NSF CAREER Grant DMS-0745185, Arnold O. Beckman Research Award (UIUC Campus Research Board 13039) and Marie Curie FP7-PEOPLE-2012-IIF 327763.} \and
Ping Hu\thanks{Department of Mathematics, University of Illinois, Urbana, IL 61801, USA, E-mail: {\tt pinghu1@math.uiuc.edu}.}  \and
Bernard Lidick\'{y}\thanks{Department of Mathematics, Iowa State University, Ames, IA 50011, USA {\tt lidicky@iastate.edu}. Research partially done at University of Illinois, Urbana.} \and
Florian Pfender\thanks{University of Colorado Denver, Mathematical and Statistical Sciences, Denver, CO, 80202, USA.
E-mail: {\tt florian.pfender@ucdenver.edu}.
Research is partially supported by
a grant from the Simons Foundation (\# 276726)}
\and
Jan Volec\thanks{
 Mathematics Institute and DIMAP, University of Warwick, Coventry CV4 7AL, UK, {\tt honza@ucw.cz}.
  Research is partially supported by the European Research Council under the European
  Union's Seventh Framework Programme (FP7/2007-2013)/ERC grant agreement no.~259385.
 } \and
Michael Young\thanks{Department of Mathematics, Iowa State University, Ames, IA 50011, USA {\tt myoung@iastate.edu}. Research supported in part by NSF grant DMS-0946431.}}

\newcommand{\Xmin}{0.244287}
\newcommand{\Xmax}{0.255713}
\newcommand{\XXmax}{0.506597}
\newcommand{\XXmin}{0.493403}
\newcommand{\Xzeromax}{0.0059605}
\newcommand{\funky}{0.000084609}
\newcommand{\funkyDeg}{0.049995} 
\newcommand{\weirdtwomin}{0.484987} 
\maketitle

\begin{abstract}
Erd\H os and S\'os proposed a problem of determining the maximum number $F(n)$ of rainbow
triangles in 3-edge-colored complete graphs on $n$ vertices.
They conjectured that $F(n)=F(a)+F(b)+F(c)+F(d)+abc+abd+acd+bcd$,
where $a+b+c+d=n$ and $a,b,c,d$ are as equal as possible. 
We prove that the conjectured recurrence holds for sufficiently large $n$.
We also prove the conjecture for $n = 4^k$ for all $k \geq 0$. These results imply that $\lim \frac{F(n)}{{n\choose 3}}=0.4$, and determine the unique limit object.
In the proof we use flag algebras combined with stability arguments.
\end{abstract}

\section{Introduction}\label{intro}
An edge-coloring of a graph (or a subgraph of a graph) is \emph{rainbow} if each of its edges has a different color. 
Let $G$ be a $3$-edge-colored $K_n$, we define $F(G)$ to be the number of rainbow triangles in $G$, and define 
\[
F(n) = \max_{G: \textrm{ $3$-edge-colored }K_n} F(G).
\] 
The following conjecture on $F(n)$ was mentioned in~\cite{EH72} as an older problem of Erd\H os and S\'os and it was mentioned again in~\cite{NesR}.
\begin{conjecture}
\begin{equation}\label{conj}
F(n)=F(a)+F(b)+F(c)+F(d)+abc+abd+acd+bcd,\end{equation}
where $a+b+c+d=n$ and $a,b,c,d$ are as equal as possible.
\end{conjecture}
This recursive formula arises from the following construction. Denote by $\RB1111$ a $3$-edge-colored $K_4$, if it has the - up to isomorphism - unique coloring that every triangle in it is  rainbow.
\begin{construction}\label{construction}
Fix an $\RB1111$, and blow up its four vertices into four classes, of sizes   $a,b,c,d$. The edges between two classes should inherit the color of the edge from the starting $\RB1111$. This way, each of the triangles having vertices in three different classes will be rainbow.
Inside of each class place an extremal coloring of $K_a,K_b,K_c,K_d$, see Figure~\ref{fig-construction}.
\end{construction}
A slight strengthening of Conjecture~\ref{conj} is as follows.
\begin{conjecture}\label{conj2}
For every $n$, all $3$-colorings of $K_n$ attaining $F(n)$ are attained via Construction~\ref{construction}.
\end{conjecture}
Up to a permutation of the colors in each iterative step, this construction gives a unique candidate for an extremal $3$-coloring of all edges of $K_n$. Note that for $n=4^k$, the allowed color permutations in each step are in fact isomorphisms, so in this case the extremal coloring is conjectured to be unique up to isomorphism.
\begin{figure}[ht]
\begin{center}
  \includegraphics[page=1,scale=1.2]{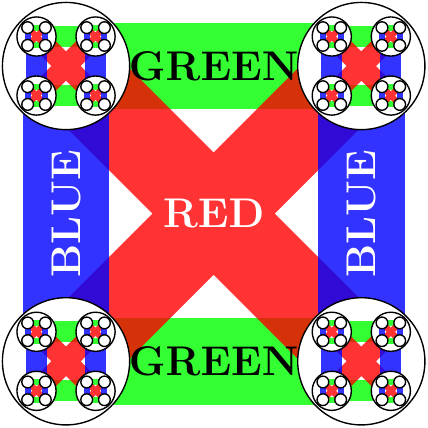}
\end{center}
\caption{Sketch of conjectured extremal construction $G^{\square}$.}\label{fig-construction}
\end{figure}
In this paper, we prove Conjecture~\ref{conj2} for large enough $n$, and for $n=4^k$ for all $k\ge 0$. 
\begin{theorem}\label{thmrecurs}
There exists $n_0$ such that 
for every  $n > n_0$
\begin{equation}
F(n)=F(a)+F(b)+F(c)+F(d)+abc+abd+acd+bcd,\end{equation}
where $a+b+c+d=n$ and $a,b,c,d$ are as equal as possible.

Moreover, if $G$ is a $3$-edge-colored graph on $n$ vertices containing $F(n)$
rainbow triangles, then $V(G)$ can be partitioned into four sets $X_1, X_2, X_3$ and $X_4$ of sizes $a, b, c$ and $d$ respectively, such that the edges containing vertices from different classes are colored like in a blow-up of a properly $3$-edge-colored $K_4$, where vertices of the $K_4$ are blown-up by $a,b,c$ and $d$ vertices.
\end{theorem}
\begin{theorem}\label{thm4k}
Conjecture~\ref{conj2} holds for $n = 4^k$, where $k \geq 1$. Moreover, the unique
extremal example is the $(k-1)$-times iterated blow-up of $\RB1111$.
\end{theorem}
We are not able to prove Conjecture~\ref{conj2} for all smaller $n$ which are not powers of $4$.
Nevertheless, Theorem~\ref{thmrecurs} is strong enough to directly imply
the uniqueness of the extremal limit homomorphism (in the flag algebra sense),
and thus the asymptotic density of rainbow triangles.

\begin{theorem}\label{mainthm}
The unique limit homomorphism maximizing the density of rainbow triangles is given by the sequence of the iterated blow-ups of $\RB1111$. This implies that
\[
\lim_{n\to\infty}\frac{F(n)}{{n\choose 3}} = 0.4.
\]
\end{theorem}

Counting the number of rainbow copies of given subgraphs was studied earlier, see for example~\cite{BaloghDLP:2013+} on a similar problem on hypercubes.
Another natural question about triangles in $3$-colored complete graphs, determining the {\it minimum} number of the monochromatic triangles,
 was solved in \cite{CummingsKPSTY:2012}.

One of the tools we use to prove Theorem~\ref{mainthm} are flag algebras.
The tool was introduced by Razborov~\cite{Razborov:2007} as a general tool to approach problems from extremal combinatorics.
%
Flag algebras have been successfully applied to various problems in extremal combinatorics.
To name some of the applications, they were used for attacking the
Caccetta-H\"aggkvist conjecture~\cite{HladkyKN:2009,RazborovCH:2011},
Tur\'an-type problems in graphs~\cite{Grzesik:2011,Hatami:2011,Nikiforov:2011,PikhurkoV:2013,Razborov:2008,Reiher:2012,Sperfeld:2011}, 
$3$-graphs~\cite{BaberT:2011,Falgas:2012,GlebovKV:2013} 
and hypercubes~\cite{Baber:2012,BaloghHLL:2014},
extremal problems in a colored environment~\cite{BaberT:2013,CummingsKPSTY:2012}, 
and also to problems in geometry~\cite{Kral:2011} or extremal theory of permutations~\cite{BaloghHLPUV:2014}.
For more details on these applications, see a recent survey of Razborov~\cite{Razborov13}.


In the case when flag algebras give a sharp bound on the density, usually the extremal structure is
`clean'. Even then, to obtain an exact result, it requires obtaining extra
information from the flag algebra computations, and then apply some stability
type  method. In most cases, this last step uses results from the computation
that certain small substructures appear with density $o(1)$.

For our problem, the conjectured extremal structure has an iterated structure,
for which it is quite rare to obtain the precise density from flag algebra
computations alone, see for example the problem on inducibility of small out-stars
in oriented graphs~\cite{Falgas:2012} (note that the problem of inducibility of all out-stars
was recently solved by Huang~\cite{Huang:2014} using different techniques).
In our case, a direct application of the semidefinite method gives only
an upper bound on the limit value and shows that $\lim_{n\to\infty}\frac{F(n)}{{n\choose 3}} < 0.40005$.
However, using flag algebras to find bounds on densities of other substructures
and combining them with other combinatorial arguments,
we manage to obtain the precise result, at least when $n$ is a power of $4$,
or when $n$ is sufficiently large.
We hope that our methods may give some insights on how to attack some other
hard problems. 


\section{Notation}
\label{sec:notation}

We say that a $3$-edge-colored graph $G$ on $n$ vertices is \emph{extremal} if $G$ contains the maximum number of rainbow triangles among all $3$-edge-colored graphs on $n$ vertices.

Given a graph $G$, we use $V(G)$ and $E(G)$ to denote its vertex set and edge set respectively, and write $v(G)=|V(G)|$. 

{\JV
Given two $3$-edge-colored 
 graphs $G$ and $G'$, an
\emph{isomorphism} between $G$ and $G'$ is a bijection $f: V(G) \to V(G')$
satisfying $\{f(v_1),f(v_2)\} \in E(G')$ if and only if $\{v_1,v_2\} \in E(G)$ and every
pair of edges $\{v_1,v_2\} \in E(G)$ and $\{f(v_1),f(v_2)\} \in E(G')$ have the same color.
Two $3$-edge-colored graphs $G$ and $G'$ are \emph{isomorphic}, which we denote
by $G \cong G'$, if and only if there exists an isomorphism between $G$ and $G'$.

In Section~\ref{flags}, we also use a coarser equivalence relation on 
$3$-edge-colored graphs, the so-called \emph{color-blind isomorphism}. We
say that two $3$-edge-colored graphs $G$ and $G'$ are color-blindly isomorphic
if there exists a permutation $\pi: \{1,2,3\}\to\{1,2,3\}$ and a bijection
$f: V(G) \to V(G')$ satisfying the following. 
A pair $\{f(v_1),f(v_2)\}$ is an edge in $G'$ if and only if $\{v_1,v_2\} \in E(G)$,
and for every edge $\{v_1,v_2\} \in E(G)$ colored by $c$ the corresponding edge
 $\{f(v_1),f(v_2)\} \in E(G')$ is colored by $\pi(c)$.
In other words, $G'$ becomes isomorphic to $G$ (in the original sense) after
renaming colors of all the edges in $G'$ according to $\pi$.

For a $3$-edge-colored graph $G$ and a vertex set $U\subseteq V(G)$, denote by
$G[U]$ the induced $3$-edge-colored subgraph of $G$ by the vertex set $U$.
For a vertex $v$ of $G$, we abbreviate $G[V\setminus\{v\}]$ to $G-v$.

Let $H$ be a $3$-edge-colored graph on $t$ vertices and $G$ be a
$3$-edge-colored graph on $n$ vertices with $n\ge t$. Denote by $P(H,G)$ the number of
$t$-subsets $U$ of $V(G)$ such that $G[U]\cong H$, and define the
\emph{density} of $H$ in $G$ to be
\[
p(H, G) = \frac{P(H, G)}{\binom{n}{t}}.
\]
In other words, $p(H, G)$ is the probability that a random subset of $V(G)$ of size $t$
induces a copy of $H$.
}

Fix a $3$-edge-colored complete graph $G$.
We denote by $\RBT$ the density of the properly $3$-edge-colored triangles, i.e., the probability
that random $3$ vertices from $G$ induce a $3$-edge-colored triangle.
Analogously, let $\TCT$ be the probability that random $3$ vertices from $G$
induce a triangle colored with exactly two colors, and $\MONOT$ the
probability that random $3$ vertices from $G$ induce a monochromatic triangle.
Note that both $\TCT$ and $\MONOT$ can be expressed as a linear combination of subgraph densities
(in fact, each of them can be expressed as a combination of three subgraph densities).
Also note that $\RBT + \TCT + \MONOT = 1$.

By $\RB1111$, we denote the density of properly $3$-edge-colored $K_4$s.
Similarly, let $\RB2111$ be the probability that random $5$ vertices from $G$
induces a $3$-edge-colored graph containing exactly two copies of $\RB1111$.
In other words, the vertices induces a $5$-vertex blow-up of $\RB1111$, where
the edge inside the unique blob of size $2$ can be colored arbitrarily.
Next, we write $\RB1111^+$ for the probability that random $5$ vertices from $G$
contains exactly one copy of $\RB1111$. 
Again, the values of $\RB2111$ and $\RB1111^+$ can be expressed as a linear combination
of subgraph densities, and it follows that $\RB1111 = 2/5 \cdot \RB2111 + 1/5 \cdot \RB1111^+$.

Finally, we define $\RB3111$ and $\RB2211$ to be the probabilities that random
$6$ vertices from $G$ induces the appropriate $6$-vertex blow-up of $\RB1111$. Specifically,
$\RB3111$ is the probability that the induced graph is obtained from $\RB1111$
by blowing-up one of its vertices twice and coloring the three edges inside
the blob arbitrarily. $\RB2211$ denotes the other option -- the probability
that we choose two different vertices of $\RB1111$ and blow-up both of them once.
See Figure~\ref{fig-D} for examples.
As in all the previous cases, both $\RB3111$ and $\RB2211$ can be expressed as an appropriate
linear combination of subgraph densities.
Hence we call any of the probabilities defined in the last three paragraphs a~\emph{density expression}.
With a slight abuse of notation, we will also use the same notation for the corresponding
classes of subgraphs.

\begin{figure}[ht]
\begin{center}
  \includegraphics[page=1]{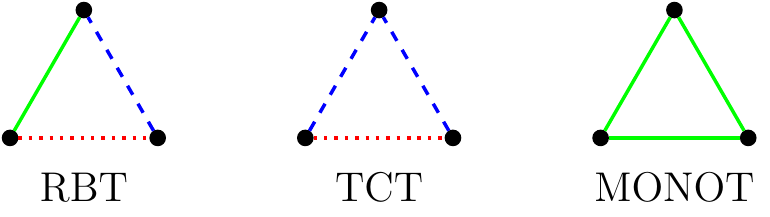}
  \vskip 0.5em
  \includegraphics[page=2]{fig-X}
\end{center}
\caption{Examples of small configurations.}\label{fig-D}
\end{figure}

Let $G$ be an extremal graph on $n$ vertices
and let $D$ be some density expression. 
For any $X\subseteq V(G)$, 
we denote by $D(X)$ the density expression $D$ restricted to subgraphs of $G$
containing $X$, and we call $D(X)$ the \emph{rooted density expression} of $D$ at $X$ in $G$.
For example,  for $X=\{x_1,x_2,x_3,x_4\}$, the rooted density expression $\RB2211(X)$ is
the probability that random $6-|X|=2$ vertices from $V(G) \setminus X$ extends $X$ to a subgraph
from $\RB2211$. Equivalently, it is the number of $\RB2211$s containing the four vertices $x_1,x_2,x_3,x_4$ divided by $\binom{n-4}{2}$. 
For a fixed vertex $u \in V(G)$, we write $D(u)$ instead of $D(\{u\})$. Similarly,
for a fixed edge $vw$, we write $D(vw)$ instead of $D(\{v,w\})$.


\section{Outline of the proof of Theorem~\ref{thmrecurs}}
The proof has some technical parts, so we give a thorough outline of the main ideas and motivations.
Theorems~\ref{thm4k} and \ref{mainthm} are consequences of Theorem~\ref{thmrecurs}, which we will prove in Section \ref{sec:thmrecurs}.
Note that the first statement in Theorem~\ref{thmrecurs} is a direct consequence of the second statement, so we only need to show the later one.
We assume that $G$ is a 3-edge-colored graph on $n$ vertices maximizing
the number of rainbow triangles.

Our first goal is to show that the vertices of $G$ can be partitioned into four sets $X_1,X_2,X_3,X_4$ of almost equal size such
that the edges between the sets look like in a blow-up of the properly 3-edge-colored $K_4$, see Figure~\ref{fig-X}. 
We start by carefully choosing a properly 3-edge-colored $K_4$ in  $G$ and use it to partition
the vertices of $G$ into sets $Z_1,\ldots,Z_4$ and a trash set $Z_0$. 
In this process we are guided by the conjectured extremal graph $G^{\square}$. In $G^{\square}$, most $\RB1111$s contain one vertex in each $X_i$.
We call an $\RB1111$ $Z$ \emph{outer} if there are at least $n/2$ vertices $v$ where
$Z + v$ forms $\RB2111$. 
Once we have found an outer $\RB1111$ (call it $Z$), adding any other vertex will result in an $\RB2111$ in $G^{\square}$. To recover the $X_i$ from $Z$, we only have to check for every vertex in $G-Z$, which of the four vertices in $Z$ is its twin.

Following this idea, we want to pick $Z$ in $G$, such that $Z$ lies in many $\RB2111$s, and determine the $Z_i$ accordingly. We can find such a $Z$ through an averaging argument from bounds given to us from some standard flag algebra computations. But just knowing a bound on the number of $\RB2111$s our set $Z$ lies in will not tell us anything about the relative sizes of the $Z_i$, so this simple approach falls short of our goal.
To remedy this problem, we look at subgraphs of size $6$ instead. Adding two vertices to $Z$ in the conjectured extremal graph gives us either an $\RB2211$ or an $\RB3111$. In $G$, the more $\RB2211$s and the fewer $\RB3111$s contain $Z$, the better the resulting sets $Z_i$ will be balanced. Thus, we look for a $Z$ maximizing
\begin{align}\label{ol1}
\RB2211(Z)-\tfrac{26}{9}\RB3111(Z),
\end{align}
where the value $\frac{26}{9}$ comes from our attempt to minimize%
\footnote{If  $\frac{26}{9}$ was replaced by 3, this function would be 0 in case all classes have the same sizes. Using a number a somewhat smaller than 3 forces the classes being more balanced.}
the gap in \eqref{xmin} from Section~\ref{sec:thmrecurs}. Again, the best we can do is to find a $Z$ which achieves at least the average of~\eqref{ol1} over all $\RB1111$. Unfortunately, the bounds on the $Z_i$ we get from this $Z$ are not quite strong enough to later push through the whole proof, so we have to work yet a little harder. Notice that in $G^{\square}$, there are also $\RB1111$s inside each of  the four $X_i$. These {\em inner} $\RB1111$s have much lower values in~\eqref{ol1}, so the average of that function is pushed down. On the other hand, if a vertex is added to an inner $\RB1111$, in most cases it results in a copy of $\RB1111^+$ and not $\RB2111$ (which are always the result when starting from an outer $\RB1111$). Following this observation, we consider the quantity
 \begin{align}\label{ol2}
\RB2211(Z)-\tfrac{26}{9}\RB3111(Z)+\tfrac{27}{1000}\RB1111^+(Z)
\end{align}
instead, where again $\frac{27}{1000}$ comes from optimizing \eqref{xmin} like $\frac{26}{9}$.
The average of~\eqref{ol2} over all $\RB1111$ is a little higher than the average of~\eqref{ol1} in $G^{\square}$, and the lower bound we get from flag algebra computations is improved as well. With this bound in hand, we can now find our $Z$ by an averaging argument, and we can guarantee that the resulting $\{Z_i\}_{i=1}^4$ are fairly balanced, and contain most vertices of $G$. 
An edge between $Z_i$ and $Z_j$ is \emph{funky} for $1\leq i <  j \leq 4$ if its color is different from what the $\RB1111$ spanned by $Z$ suggests.
There are only few funky edges, as every such edge reduces $\RB2211(Z)$.
We remove (very few) vertices incident to too many funky edges from  $Z_i$, and obtain $X_1,\ldots,X_4$ and a trash set $X_0$ of all remaining vertices, while still maintaining fairly strong bounds on the sizes of  $X_i$s.

Using this structure, we can now step by step get closer to our goal.
In Claim~\ref{c1} we show that a vertex in $X_i$ is not adjacent to almost all other vertices in $X_i$ by edges
of only one color. Otherwise, this vertex would lie in too few rainbow triangles, contradicting the simple Proposition~\ref{triDp} with the consequence that
$\RBT(v)=0.4+o(1)$ for every vertex $v$ in $G$.

The remainder of the proof uses mostly recoloring arguments; we rule out certain scenarios by showing that recoloring some edges in these scenarios would increase the number of rainbow triangles.

If some edge $uv$ is funky with $v\in X_i$, then the vast majority of the edges from $v$ to other vertices in $X_i$ must have the same color, as otherwise recoloring $uv$ would increase the number of rainbow triangles. This is stated precisely in Claim~\ref{c2}. 

The last two claims show that every vertex incident to funky edges must be incident to more than $0.4n$ edges of the same color. Using bounds from another flag algebra computation, we can show that this can occur only for very few vertices in Claim~\ref{c4}, and therefore the funky edges are incident to only a very small number of vertices. Using this knowledge, we can use a recoloring argument very similar to the one in Claim~\ref{c2}, yielding bounds contradicting Claim~\ref{c1}.
This contradiction shows that in fact there are no funky edges.

Therefore, all the edges between $X_i$ and $X_j$ have the right color but we still need to deal
with vertices in $X_0$.
In Claims~\ref{c01} and \ref{c02} we show that if we forcefully include a vertex from $X_0$ in any
$X_i$, it will result in many funky edges. In other words, every vertex in $X_0$ looks very different from vertices in the other $X_i$.
In fact, vertices in $X_0$ look so different from vertices in the $X_i$ that we can show that they cannot lie in enough rainbow triangles, so $X_0$ must be empty. 
This last argument in Claim~\ref{c03} relies on a massive case analysis handled by the computer, as we are maximizing a quadratic function over a $12$-dimensional polytope with thousands of facets. 

To complete the proof, we show in Claim~\ref{c8} that the sizes of the $X_i$ are almost balanced. 

\section{Flag algebras} \label{flags}
The aim of this section is to establish the following statement.
\begin{proposition} \label{prop:flag}  
There exists $n_0 \in \Nat$ such that every extremal $3$-edge-colored complete graph $G$ on at least $n_0$ vertices has the following properties:
\begin{align}
  \tfrac{4}{15}\RB2211 - \tfrac{26}{45}\RB3111 + \tfrac{27}{5000}\RB1111^+ & > 0.002629395;\label{main}\\
  \RBT& < 0.40005;\label{RBT}\\
  \RB1111& < 0.09523837;\label{RB1111}\\
  \tfrac13\TCT + \MONOT& < 0.33343492.\label{TCT}
\end{align}
\end{proposition}
Let us give the related subgraph densities in Construction~\ref{construction}:
\begin{align*}
\RB2211 &= 270/1023, &
\RB3111 &= 120/1023,\\
\RB1111^+ &= 2/357, &
\RBT &= 0.4,\\
\RB1111 &= 2/21,&
\TCT/3+\MONOT &= 1/3.
\end{align*}
We also list the arithmetic values of \eqref{main} to \eqref{TCT} for Construction~\ref{construction} below:
\begin{align*}
  \tfrac{4}{15}\RB2211 - \tfrac{26}{45}\RB3111 + \tfrac{27}{5000}\RB1111^+ &\approx0.002636964;\\
  \RBT&= 0.4;\\
  \RB1111&\approx 0.095238095;\\
  \tfrac13\TCT + \MONOT& \approx  0.333333333.
\end{align*}

The main tool used for the proof of~Proposition~\ref{prop:flag} is flag algebras.

\subsection{Flag algebra terminology}

Let us now introduce the terminology related to flag algebras
needed in this paper. Since we deal only with $3$-edge-colored complete graphs,
we restrict our attention just to this particular case.
The central notions we are going to introduce are an algebra $\cA$ and algebras
$\cA^{\sigma}$, where $\sigma$ is a fixed $3$-edge-coloring of a complete graph.
Let us point out that we build flag algebras here with respect to the
color-blind isomorphism instead of the standard isomorphism
of $3$-edge-colored graphs. This has been done only due to technical reasons,
specifically, it decreased the computational effort needed for proving the inequalities
in Proposition~\ref{prop:flag}. Note that all the density expressions defined in
Section~\ref{sec:notation} are invariant under permutations of the colors. Therefore,
the values of the density expressions defined in Section~\ref{sec:notation} can
be expressed as certain linear combinations of color-blind subgraph densities.

In order to precisely describe algebras $\cA$ and $\cA^\sigma$, we first need to introduce
some additional notation.
Let $\cF$ be the set of all finite $3$-edge-colored complete graphs modulo color-blind isomorphism.
Next, for every $\ell\in\Nat$, let $\cF_\ell\subset \cF$
be the set of $\ell$-vertex $3$-edge-colored graphs from $\cF$.
For $H\in\cF_\ell$ and $H'\in\cF_{\ell'}$, recall that
$p(H,H')$ is the probability that a randomly chosen subset of
$\ell$ vertices in $H'$ induces a subgraph isomorphic to $H$.
Note that $p(H,H')=0$ if $\ell' < \ell$.
Let $\Real\cF$ be the set of all formal linear combinations
of elements of $\cF$ with real coefficients. Furthermore, let $\cK$ be the linear
subspace of $\Real\cF$ generated by all linear combinations of the form
\[H-\sum_{H'\in\cF_{v(H)+1}}p(H,H')\cdot H'.\]
Finally, we define $\cA$ to be the space $\Real\cF$ factorized by $\cK$.

The space $\cA$ has naturally defined linear operations of an addition, and a multiplication by a real number.
We now introduce a multiplication inside $\cA$.
We first define it on the elements of $\cF$ in the following way. For $H_1, H_2 \in \cF$, and $H\in\cF_{v(H_1)+v(H_2)}$,
we define $p(H_1, H_2; H)$ to be the probability that a randomly chosen subset of $V(H)$
of size $v(H_1)$ and its complement induce in $H$ subgraphs color-blindly isomorphic
to $H_1$ and $H_2$, respectively.
We set
\[H_1 \times H_2 = \sum_{H\in\cF_{v(H_1)+v(H_2)}}p(H_1,H_2;H) \cdot H.\]
The multiplication on $\cF$ has a unique linear extension to $\Real\cF$, which yields
a well-defined multiplication also in the factor algebra $\cA$. A formal proof
of this can be found in~\cite[Lemma 2.4]{Razborov:2007}.

Let us now move to the definition of an algebra $\cA^\sigma$, where $\sigma \in
\cF$ is an arbitrary $3$-edge-colored complete graph with a fixed labelling of its vertex
set. The labelled graph $\sigma$ is usually called a~{\em type} within the flag algebra
framework.
Without loss of generality, we will assume that the vertices of $\sigma$ are labelled by $1,2,\dots,v(\sigma)$.
Now we follow almost the same lines as in the definition of $\cA$.
We define $\cF^{\sigma}$ to be the set of all finite $3$-edge-colored complete graphs $H$ with a fixed {\em embedding} of
$\sigma$, i.e., an injective mapping $\theta$ from $V(\sigma)$ to $V(H)$ such that
$\im(\theta)$ induces in $H$ a subgraph isomorphic to $\sigma$.
Again, the graphs in $\cF^\sigma$ are considered modulo color-blind isomorphism.
The elements of $\cF^{\sigma}$ are usually called \emph{$\sigma$-flags}
and the subgraph induced by $\im(\theta)$ is called the \emph{root} of a $\sigma$-flag.

Again, for every $\ell\in\Nat$, we define
$\cF^{\sigma}_\ell\subset \cF^{\sigma}$ to be the set of the $\sigma$-flags from
$\cF^{\sigma}$ that have size $\ell$ (i.e., the $\sigma$-flags with the underlying $3$-edge-colored graph having $\ell$ vertices).
Analogously to the case for $\cA$, for two $3$-edge-colored graphs $H, H' \in\cF^{\sigma}$ with the embeddings of $\sigma$ given by $\theta, \theta'$, we set
$p(H,H')$ to be the probability that a randomly chosen subset of $v(H)-v(\sigma)$ vertices in
$V(H')\setminus\theta'(V(\sigma))$ together with $\theta'(V(\sigma))$ induces a
subgraph that is color-blindly isomorphic to $H$ through an isomorphism $f$ that preserves the embedding of $\sigma$.
In other words, the color-blind isomorphism $f$ has to satisfy $f(\theta') = \theta$.
Let $\Real\cF^{\sigma}$ be the set of all formal linear combinations of elements
of $\cF^\sigma$ with real coefficients, and let $\cK^\sigma$ be the linear subspace
of $\Real\cF^\sigma$ generated by all the linear combinations of the form
\[H-\sum_{H'\in\cF^\sigma_{v(H)+1}}p(H,H')\cdot H'.\]
We define $\cA^\sigma$ to be $\Real\cF^\sigma$ factorised by $\cK^\sigma$.

We now describe the multiplication of two elements from $\cF^\sigma$. 
Let $H_1, H_2\in \cF^\sigma$, \\ $H\in \cF^\sigma_{v(H_1)+v(H_2)-v(\sigma)}$, and $\theta$ be the fixed embedding of $\sigma$ in $H$.
As in the definition of multiplication for $\cA$, we define $p(H_1, H_2; H)$ to be the probability that
a randomly chosen subset of $V(H)\setminus \theta(V(\sigma))$ of size
$v(H_1)-v(\sigma)$ and its complement in $V(H)\setminus \theta(V(\sigma))$ of
size $v(H_2)-v(\sigma)$, extend $\theta(V(\sigma))$ in $H$ to subgraphs
color-blindly isomorphic to $H_1$ and $H_2$, respectively.  Again, by
isomorphic here we mean that there is a color-blind isomorphism that preserves
the fixed embedding of $\sigma$. This definition naturally extends to $\cA^\sigma$.

Now consider an infinite sequence $(G_n)_{n\in\Nat}$ of $3$-edge-colored complete graphs of increasing orders.
We say that the sequence $(G_n)_{n\in\Nat}$ is \emph{convergent} if the probability $p(H,G_n)$ has a limit for every $H\in\cF$.
A standard compactness argument (e.g., using Tychonoff's theorem) 
yields that every such infinite sequence has a convergent subsequence.
All the following results can be found in~\cite{Razborov:2007}.
Fix a convergent increasing sequence $(G_n)_{n\in\Nat}$ of $3$-edge-colored graphs.
For every $H\in\cF$, we set $\phi(H) = \lim_{n\to\infty} p(H,G_n)$ and linearly extend $\phi$ to $\cA$.
We usually refer to the mapping $\phi$ as to the limit of the sequence.
The obtained mapping $\phi$ is a homomorphism from $\cA$ to $\Real$.
Moreover, for every $H\in \cF$, we obtain $\phi(H)\geq 0$. 
Let $\Hom^+(\cA, \Real)$ be the set of all such homomorphisms, i.e., the set of all homomorphisms
$\psi$ from the algebra $\cA$ to $\Real$ such that $\psi(H)\ge0$ for every $H\in\cF$.
It is interesting to see that this set is exactly the set of all limits of convergent sequences of $3$-edge-colored complete
graphs~\cite[Theorem~3.3]{Razborov:2007}.

Let $(G_n)_{n\in\Nat}$ be a convergent sequence of $3$-edge-colored graphs and $\phi \in \Hom^+(\cA, \Real)$ be its limit.
For $\sigma\in\cF$ and an embedding $\theta$ of $\sigma$ in $G_n$,
we define $G_n^\theta$ to be the $3$-edge-colored graph rooted on the copy of $\sigma$ that corresponds to $\theta$.
For every $n\in\Nat$ and $H^\sigma \in \cF^\sigma$, we define 
$p^\theta_n(H^\sigma)=p(H^\sigma,G_n^\sigma)$.
Picking $\theta$ at random gives rise to a probability distribution ${\bf P}_{\bf n}^\sigma$ on mappings
from $\cA^{\sigma}$ to $\Real$, for every $n\in\Nat$.
Since $p(H,G_n)$ converges (as $n$ tends to infinity) for every $H\in\cF$,
the sequence of these probability distributions on mappings from $\cA^{\sigma}$ to $\Real$ also converges~\cite[Theorems 3.12 and 3.13]{Razborov:2007}.
We denote the limit probability distribution by ${\bf P}^\sigma$.
In fact, for any $\sigma$ such that $\phi(\sigma) > 0$, the homomorphism $\phi$ itself fully determines the random distribution ${\bf P}^\sigma$~\cite[Theorem 3.5]{Razborov:2007}.
Furthermore, any mapping $\phi^\sigma$ from the support of the distribution ${\bf P}^\sigma$ is in fact a homomorphism from
$\cA^{\sigma}$ to $\Real$ such that $\phi^\sigma(H^\sigma) \ge 0$ for all $H^\sigma \in \cF^\sigma$~\cite[Proof of Theorem 3.5]{Razborov:2007}.

The last notion we introduce is the \emph{averaging} (or downward) operator
$\llbracket\cdot\rrbracket_{\sigma}: \cA^{\sigma} \to \cA$. It is a linear operator
defined on the elements of $H^\sigma \in \cF^\sigma$ by $\llbracket{H^\sigma}\rrbracket_{\sigma} = p_H^\sigma \cdot H^\emptyset$, where
$H^\emptyset$ is the (unlabeled) $3$-edge-colored graph from $\cF$ corresponding to $H^\sigma$,
and $p_H^\sigma$ is the probability that a random injective mapping
from $V(\sigma)$ to $V(H^\emptyset)$ is an embedding of $\sigma$ in $H^\emptyset$ yielding a $\sigma$-flag color-blindly isomorphic to $H^\sigma$.
The key relation between $\phi$ and $\phi^\sigma$ is the following:
\[
\forall H^\sigma\in\cA^\sigma,\quad \phi\left(\llbracket{H^\sigma}\rrbracket_{\sigma}\right)=\phi(\llbracket\sigma\rrbracket_\sigma) \cdot \int \phi^\sigma(H^\sigma),
\]
where the integration is over the probability space given
by the random distribution ${\bf P}^\sigma$ on $\phi^\sigma$.
Therefore, if $\phi^\sigma(A^\sigma)\ge 0$ almost surely for some $A^\sigma \in \cA^\sigma$,
then $\phi\left(\left\llbracket{A^\sigma}\right\rrbracket_{\sigma}\right)\ge 0$.
In particular,
\begin{equation}
\label{eq:cauchyschwarz}
\forall A^\sigma\in\cA^\sigma,\quad \phi\left(\left\llbracket{\left(A^\sigma\right)^2}\right\rrbracket_{\sigma}\right)\ge 0.
\end{equation}

The semidefinite method is a tool from the flag algebra framework that, for a
given density problem of the form \[\min_{\phi\in\Hom^+(\cA,\Real)}\phi(A),\]
where $A\in\cA$, systematically searches for `best possible' inequalities
of the form~(\ref{eq:cauchyschwarz}). If we fix in advance an upper bound on
the size of graphs in the terms of inequalities we will be using, we can
find the best inequalities of the form~(\ref{eq:cauchyschwarz}) using
semidefinite programming. Furthermore, it is easy to extend this basic
semidefinite method in such a way that together with
inequalities~(\ref{eq:cauchyschwarz}), it uses also inequalities from a given
finitely-dimensional linear subspace of $\cA$.




\subsection{Proof of Proposition~\ref{prop:flag}}
We start this section by showing that in an extremal graph, every two vertices participate in almost the same number of rainbow triangles.
\begin{proposition}\label{triDp}
In an extremal graph $G$ on $n$ vertices, for any pair of vertices $u,v\in V(G)$, we have $\binom{n-1}{2}(\RBT(u)-\RBT(v))\le n-2$.
\end{proposition}
\begin{proof}
Otherwise, we could delete $v$ and duplicate $u$ to $u'$, i.e., for every vertex $x$ we could color the edge $xu'$ as $xu$.  This implies that the color of $uu'$ does not matter since $uu'$ will not be in a rainbow triangle anyways. Let us call the new graph $G'$.
Then
\begin{align*}
F(G') - F(G)&\ge{n-1\choose 2}(\RBT(u)-\RBT(v))-{n-2\choose 1}\RBT(uv)\\
&\ge {n-1\choose 2}(\RBT(u)-\RBT(v))-(n-2)>0,
\end{align*}
a contradiction.
\end{proof}
Combining this with the bound given by the iterative construction depicted in Figure~\ref{fig-construction} yields
the following.
\begin{corollary}\label{triD}
In an extremal graph $G$, $\RBT(v)\ge 0.4-o(1)$ for all vertices $v\in V(G)$.
\end{corollary}

Let $(E_n)_{n\in\Nat}$ be any convergent sequence of extremal graphs of increasing orders with
$e\in\Hom^+(\cF,\Real)$ being its limit. We call such $e$ an extremal limit.
We now look at the additional properties that every extremal limit needs to satisfy.
We start with a ``flag algebra version'' of Corollary~\ref{triD}.
\begin{corollary}
Let $\sigma$ be the $1$-vertex type, $\RBT^\sigma$ be the $\sigma$-flag of size three 
with all three edges colored differently (which is unique up to color-blind isomorphism),
$e$ be an extremal limit and $e^\sigma$ be a random homomorphism drawn from ${\bf P}^\sigma$ of $e$.
Then with probability $1$,
\[
e^\sigma\left(\RBT^\sigma - 1/4\right) \ge 0.
\]
Furthermore, for any real $w \ge 0$ and $F^\sigma\in\cF^\sigma$, it follows that
\begin{equation}
\label{eq:regularity}
e\left(w\cdot\left\llbracket{F^\sigma \times \left(\RBT^\sigma - 1/4\right)}\right\rrbracket_{\sigma}\right)\ge 0.
\end{equation}
\end{corollary}
Next, we apply four times the semidefinite method that seeks for inequalities of the form~(\ref{eq:cauchyschwarz}) and~(\ref{eq:regularity})
to conclude the following.
\begin{lemma}
\label{lem:flag}
For every extremal limit $e$:
\begin{align*}
  e\left(\tfrac{4}{15}\RB2211 - \tfrac{26}{45}\RB3111 + \tfrac{27}{5000}\RB1111^+\right) &\geq \flagsboundA;\\
  e\left(\RBT\right) &\leq \flagsboundB  ;\\
  e\left(\RB1111\right) &\leq \flagsboundC ;\\
  e\left(\tfrac13\TCT + \MONOT\right) &\leq \flagsboundD.
\end{align*}
\end{lemma}
\begin{proof}
At the beginning, we express all four left-hand sides as a linear combination of densities of graphs on $6$ vertices.
Note that $|\cF_6|=4300$.

The first inequality can be obtained as the sum of the following inequalities:
\begin{itemize}
\item 163 inequalities of the form
$e\left(\left\llbracket\left( \sum_{F\in\cF_5^{\sigma}} {x_F\cdot F} \right)^2\right\rrbracket_\sigma \right) \ge 0$,
where $\sigma$ is a (not always the same) type of on $4$ vertices and $x_F \in \Rat$ for all $F\in\cF^\sigma_5$,

\item  14 inequalities of the form
$e\left(\left\llbracket\left( \sum_{F\in\cF_4^{\sigma}} {x_F\cdot F} \right)^2\right\rrbracket_\sigma\right) \ge 0$,
where $\sigma$ is the only  $2$-vertex type (up to the blind-isomorphism) and $x_F \in \Rat$ for all $F\in\cF^\sigma_4$,

\item one inequality   of the form $e\left(\left( \sum_{F\in\cF_3} {x_F\cdot F} \right)^2\right) \ge 0$,
where $x_F \in \Rat$ for all $F\in\cF_3$,

\item  17 inequalities of the form $e\left(w\cdot\left\llbracket{F \times \left(\RBT^\sigma - 1/4\right)}\right\rrbracket_{\sigma}\right) \ge 0$,
where $\sigma$ is the $1$-vertex type, $w\ge0$ and $F \in \cF_4^\sigma$,

\item an inequality of the form $e\left(\sum_{F\in\cF_6}y_F\cdot F\right) \ge 0$, where $y_F \ge 0$ for all $F\in\cF_6$,

\item the equation $e\left(z \cdot \sum_{F_i\in\cF_6}F_i \right) = z$, where $z = \flagsboundA$.
\end{itemize}

\medskip

The second inequality can be obtained as the sum of the following inequalities:
\begin{itemize}
\item 884 inequalities of the form
$e\left(-\left\llbracket\left( \sum_{F\in\cF_5^{\sigma}} {x_F\cdot F} \right)^2\right\rrbracket_\sigma \right) \le 0$,
where $\sigma$ is a (not always the same) type of on $4$ vertices and $x_F \in \Rat$ for all $F\in\cF^\sigma_5$,

\item  30 inequalities of the form
$e\left(-\left\llbracket\left( \sum_{F\in\cF_4^{\sigma}} {x_F\cdot F} \right)^2\right\rrbracket_\sigma\right) \le 0$,
where $\sigma$ is the only  $2$-vertex type (up to the blind-isomorphism) and $x_F \in \Rat$ for all $F\in\cF^\sigma_4$,

\item an inequality of the form $e\left(-\sum_{F\in\cF_6}y_F\cdot F\right) \le 0$, where $y_F \ge 0$ for all $F\in\cF_6$,

\item the equation $e\left(z \cdot \sum_{F_i\in\cF_6}F_i \right) = z$, where $z = \flagsboundB$.
\end{itemize}

\medskip

The third inequality can be obtained as the sum of the following inequalities:
\begin{itemize}
\item 948 inequalities of the form
$e\left(-\left\llbracket\left( \sum_{F\in\cF_5^{\sigma}} {x_F\cdot F} \right)^2\right\rrbracket_\sigma \right) \le 0$,
where $\sigma$ is a (not always the same) type of on $4$ vertices and $x_F \in \Rat$ for all $F\in\cF^\sigma_5$,

\item  38 inequalities of the form
$e\left(-\left\llbracket\left( \sum_{F\in\cF_4^{\sigma}} {x_F\cdot F} \right)^2\right\rrbracket_\sigma\right) \le 0$,
where $\sigma$ is the only  $2$-vertex type (up to the blind-isomorphism) and $x_F \in \Rat$ for all $F\in\cF^\sigma_4$,

\item  15 inequalities of the form $e\left(-w\cdot\left\llbracket{F \times \left(\RBT^\sigma - 1/4\right)}\right\rrbracket_{\sigma}\right) \le 0$,
where $\sigma$ is the $1$-vertex type, $w\ge0$ and $F \in \cF_4^\sigma$,

\item an inequality of the form $e\left(-\sum_{F\in\cF_6}y_F\cdot F\right) \le 0$, where $y_F \ge 0$ for all $F\in\cF_6$,

\item the equation $e\left(z \cdot \sum_{F_i\in\cF_6}F_i \right) = z$, where $z = \flagsboundC$.
\end{itemize}

\medskip

Finally, the last inequality can obtained as the sum of the following inequalities:
\begin{itemize}
\item 876 inequalities of the form
$e\left(-\left\llbracket\left( \sum_{F\in\cF_5^{\sigma}} {x_F\cdot F} \right)^2\right\rrbracket_\sigma \right) \le 0$,
where $\sigma$ is a (not always the same) type of on $4$ vertices and $x_F \in \Rat$ for all $F\in\cF^\sigma_5$,

\item  34 inequalities of the form
$e\left(-\left\llbracket\left( \sum_{F\in\cF_4^{\sigma}} {x_F\cdot F} \right)^2\right\rrbracket_\sigma\right) \le 0$,
where $\sigma$ is the only  $2$-vertex type (up to the blind-isomorphism) and $x_F \in \Rat$ for all $F\in\cF^\sigma_4$,

\item  21 inequalities of the form $e\left(-w\cdot\left\llbracket{F \times \left(\RBT^\sigma - 1/4\right)}\right\rrbracket_{\sigma}\right) \le 0$,
where $\sigma$ is the $1$-vertex type, $w\ge0$ and $F \in \cF_4^\sigma$,

\item an inequality of the form $e\left(-\sum_{F\in\cF_6}y_F\cdot F\right) \le 0$, where $y_F \ge 0$ for all $F\in\cF_6$,

\item the equation $e\left(z \cdot \sum_{F_i\in\cF_6}F_i \right) = z$, where $z = \flagsboundD$.
\end{itemize}

The exact rational values of all the coefficients $x_F, y_F$ and $w$ that
appears in the inequalities above were obtained with computer assistance.
They are available at \oururl, as well as a small Sage script that computes the
corresponding sums.
\end{proof}

In order to prove Proposition~\ref{prop:flag}, we just translate the previous statement back to the finite setting.
\begin{proof}[Proof of Proposition~\ref{prop:flag}]
Suppose one of the inequalities from the statement of Proposition~\ref{prop:flag} is false.
For example, suppose that the inequality~\eqref{RBT} is false. Therefore, for every $k\in\Nat$ we can find
an extremal graph $E_k$ on at least $k$ vertices such that $\RBT \ge 0.40005$. By compactness,
the sequence $(E_k)_{k\in\Nat}$ has a convergent subsequence and this subsequence converges
to some extremal limit $e$. 
However, $e(\RBT) \ge 0.40005$, which contradicts Lemma~\ref{lem:flag}.
\end{proof}

\section{Proof of Theorem~\ref{thmrecurs}}\label{sec:thmrecurs}

Let $G$ be an extremal graph on $n$ vertices, where $n$ is sufficiently large.
 Let $Z=\{z_1,z_2,z_3,z_4\}$ be a subset of $V(G)$ such that $Z$ induces an $\RB1111$, and
\begin{align}
\RB2211(Z) - \tfrac{26}{9}\RB3111(Z) + \tfrac{27}{1000}\RB1111^+(Z) \label{firstZ}
\end{align}
is maximized over all choices of $Z$. 

Note that in every $\RB2211$, four of the $15$ vertex subsets of size $4$ induce copies of $\RB1111$, three in every $\RB3111$, and one of the five sets in every $\RB1111^+$.
Denote by $\mathcal Z$ the set of all properly 3-edge-colored $K_4$s. 
Since \eqref{firstZ} is maximized, we can lower bound it by the average  over all $Y\in \mathcal Z$ and we obtain
\begin{align}
 & \left(\RB2211(Z) - \tfrac{26}{9}\RB3111(Z) + \tfrac{27}{1000}\RB1111^+(Z)\right)\binom{n-4}{2} \nonumber \\
\geq~& \frac{1}{|\mathcal Z|} \sum_{Y \in \mathcal Z}\left(\left(\RB2211(Y) - \tfrac{26}{9}\RB3111(Y)\right)\binom{n-4}{2} + \tfrac{27}{2000}\RB1111^+(Y)\binom{n-4}{1}(n-5) \right)\nonumber\\
\geq~& \frac{\left( 4\RB2211 - 3\cdot\tfrac{26}{9}\RB3111\right)\binom{n}{6} + \tfrac{27}{2000}\RB1111^+\binom{n}{5}(n-5)}{\RB1111\binom{n}{4}} \nonumber \\ 
=~&\frac{\tfrac{4}{15}\RB2211 - \tfrac{26}{45}\RB3111 + \tfrac{27}{5000}\RB1111^+}{\RB1111} \binom{n-4}{2}. \nonumber
\end{align}

\noindent Using \eqref{main} and \eqref{RB1111}, this gives
\begin{align}
\RB2211(Z) - \tfrac{26}{9}\RB3111(Z) + \tfrac{27}{1000}\RB1111^+(Z) 
> 0.02760856.\label{main2}
\end{align}

\noindent For $1\le i\le 4$, we define sets of vertices $Z_i$ which look like $z_i$
to the other vertices of $Z$. Formally,
\[
Z_i:=\{v\in V(G):G[(Z\setminus z_i)\cup v] \cong \RB1111\}\text{ for } 1\le i\le 4.
\]
Note that $Z_i\cap Z_j=\emptyset$ for $i\ne j$. 
We call an edge $v_iv_j$ {\em funky}, if the color of $v_iv_j$ is different from the color of $z_iz_j$, where $v_i\in Z_i$, $v_j\in Z_j$, $1\le i<j\le 4$.
In other words, $G[Z\cup\{v_i,v_j\}]\ncong \RB2211$, i.e., every funky edge destroys a potential copy of $\RB2211(Z)$. Denote by $E_f$ the set of funky edges. With this notation, for sufficiently large $n$ \eqref{main2} implies that
\[
2\sum_{1\le i<j\le 4}|Z_i||Z_j|-2|E_f|-\tfrac{26}{9}\sum_{1\le i\le 4}|Z_i|^2+\tfrac{27n}{1000}\left( n-\sum_{1\le i\le 4}|Z_i|\right) >0.02760856\times2\binom{n-4}{2}.
\]
For $X_i\subseteq Z_i$, where $1\le i\le 4$, let $X_0:=V(G)\setminus\bigcup X_i$. Let $f$ be the number of funky edges not incident to vertices in $X_0$, divided by $n^2$ for normalization, and denote $x_i =\tfrac1n|X_i|$ for $0 \leq i \leq 4$. Choose $X_i$s such that the left hand side of
\begin{align}
2\sum_{1\le i<j\le 4}x_ix_j-2f-\tfrac{26}{9}\sum_{1\le i\le 4}x_i^2+\tfrac{27}{1000}x_0 >0.02760856\label{main3}
\end{align}
is maximized.

From this, it is not difficult to check (see Appendix~\ref{ax}) that
\begin{align}
x_0&<\Xzeromax;\label{xzero}\\
\Xmin<~x_i&<\Xmax\quad\text{for}\quad1\le i\le 4;\label{xmin}\\
\XXmin<~x_i+x_j&<\XXmax\quad\text{for}\quad1\le i<j\le 4;\label{xxmin}\\
f&<\funky;\label{fmax}\\
-\tfrac{25}{27}x_1+2x_i-\tfrac13x_2-\tfrac13x_3-\tfrac13x_4&<0.0315\quad\text{for}\quad2\le i\le 4;\label{weird}\\
2x_1-x_2+x_3-x_0&>\weirdtwomin ;&\label{weird2}\\
x_i+x_0 &<0.2563\quad\text{for}\quad1\le i\le 4.\label{sumxxzero}
\end{align}
By symmetry, \eqref{weird} and \eqref{weird2} hold also after permuting the variables. 
However, we use them explicitly only in this permutation.
Furthermore, for any vertex $v\in X_i$ we use $d_f(v)$ to denote the number of funky edges from $v$ to $(X_1\cup X_2\cup X_3\cup X_4)\setminus X_i$ after normalizing by $n$. 
The contribution of $v\in X_1$ to \eqref{main3} is
\[
\frac{1}{n}\left(2(x_2+x_3+x_4) - 2d_f(v)-2\cdot\tfrac{26}{9}x_1 + o(1)\right).
\]
If this quantity was negative, \eqref{main3} could be increased by moving $v$ to $X_0$, contradicting our choice of  $X_i$.
This and \eqref{xmin} imply that
\begin{align}
 d_f(v) \leq x_2+x_3+x_4-\tfrac{26}{9}x_1+o(1) \leq 1 - \tfrac{35}{9}x_1+o(1)  < \funkyDeg,\label{maxfunky}
\end{align}
and symmetric statements hold for $v\in X_2,X_3,X_4$.

By symmetry, we may assume that the non-funky edges are colored as in Figure~\ref{fig-X}.
\begin{figure}[ht]
\begin{center}
  \includegraphics[page=3,scale=1.2]{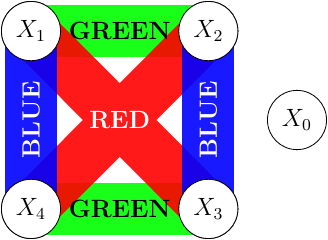}
\end{center}
\caption{Coloring of the non-funky edges. 
}\label{fig-X}
\end{figure}

Next, we will prove that a vertex $v\in X_i$ cannot be adjacent to almost all vertices of $X_i$ by edges of only one color. For a vertex $v\in V(G)$, we denote by $r_i(v)$, $b_i(v)$ and $g_i(v)$ the numbers of red, blue and green edges from $v$ to $X_i$, divided by $n$. Similarly, let $r(v)$, $b(v)$, and $g(v)$ be the numbers of all red/blue/green edges incident to $v$, divided by $n$.
\begin{claim}\label{c1}
For every $v\in X_i$, we have $x_i-r_i(v), x_i-b_i(v), x_i-g_i(v) >0.033$, where $i \in \{1,2,3,4\}$.
\end{claim}
\begin{proof}
Without loss of generality, let us assume $v\in X_1$ and $x_1-r_1(v)\le 0.033$.
Denote $x_{\max}:=\max\{x_2,x_3,x_4\}$.
We bound the number of rainbow triangles containing $v$ divided by $n^2$, i.e., $\frac12\RBT(v)$.
For a rainbow triangle $uvw$, we distinguish several cases. 

\setlist[enumerate,1]{leftmargin=0.5cm}
\begin{enumerate}[noitemsep,topsep=0pt,parsep=0pt,partopsep=0pt]
\item If $u,w \in X_1$, then 
the normalized number of rainbow triangles $uvw$ can be upper bounded by $r_1(v)b_1(v) + r_1(v)g_1(v) + g_1(v)b_1(v)$.
This is maximized when $g_1(v)=b_1(v)=\tfrac{1}{2}(x_1-r_1(v))$, which gives the upper bound
$(r_1(v)+\tfrac{x_1-r_1(v)}{4})(x_1-r_1(v))$ for triangles of this type.
\item If $u \in X_i$ and $w \in X_j$, where $1<i<j\leq 4$,  and all of $uv, vw, uw$ are non-funky,
then we obtain the upper bound
$x_2x_3+x_2x_4+x_3x_4-d_f(v)(x_2+x_3+x_4-x_{\max})+\frac13d_f(v)^2$ for triangles of this type, where the third term accounts for possible double counting in the second term. 
\item If $uw$ is a funky edge, then $uvw$ might be rainbow and in this
case we get the upper bound $f$ for triangles of this type. 
\item If $u\in X_0$ then $w$ can be anywhere,
which gives the bound $x_0$  for triangles of this type.
\item We can bound the number of rainbow triangles where both $vu$ and $vw$ are funky by $\frac{1}{3}d_f(v)^2$. The $\frac13$ in the term comes from the fact that $vu$ and $vw$ must have different colors for the triangle to be rainbow.
\item If $vu$ is funky and $w\in X_1$, then we get an upper bound of $d_f(v)r_1(v)$  for triangles of this type.
\item If $vu$ is funky and $v$ and $w$ are in the same $X_i$ (for $i\ge 2$), we get an upper bound of $d_f(v)x_{\max}$  for triangles of this type.
\end{enumerate}
Note that it cannot happen that only $vu$ is funky, $v \in X_i$, and $w \in X_j$, where $i,j \in \{2,3,4\}$ and $i \neq j$.

Counting all types together, we obtain
\begin{align}
\tfrac12 \RBT(v)\le&~\left(r_1(v)+\tfrac{x_1-r_1(v)}{4}\right)(x_1-r_1(v))
+x_2x_3+x_2x_4+x_3x_4\nonumber\\
&~+f+x_0+d_f(v)(2x_{\max}+r_1(v)-x_2-x_3-x_4+\tfrac23d_f(v)
)\label{eq:rbt} 
<~0.1991,
\end{align}
which contradicts Corollary~\ref{triD}. The last inequality can be obtained 
by maximizing \eqref{eq:rbt} in the following way.

If $x_1-r_1(v)\le 0.033$, then $r_1(v) \geq \Xmin - 0.033$ and the partial derivative of the right hand side of~\eqref{eq:rbt} in direction $r_1(v)$ is
$\tfrac{3}{4}x_1 - \tfrac{3}{2}r_1(v) + d_f(v)$, which is
 negative. Thus, to maximize the bound, we need to pick $r_1(v)$ minimal, and thus we may assume that $x_1-r_1(v)=0.033$.
 
Next, we get that the coefficient of $d_f(v)$ in \eqref{eq:rbt} is
\begin{align*}
&~~~~2x_{\max}+r_1(v)-x_2-x_3-x_4+\tfrac23d_f(v) =x_1+2x_{\max}-x_2-x_3-x_4-0.033+\tfrac23d_f(v)\\
&\le_{\eqref{maxfunky}} x_1+2x_{\max}-x_2-x_3-x_4-0.033+\tfrac23(x_2+x_3+x_4-\tfrac{26}{9}x_1+o(1)) \\
&=-\tfrac{25}{27}x_1+2x_{\max}-0.033-\tfrac13x_2-\tfrac13x_3-\tfrac13x_4+o(1)
<_{\eqref{weird}}0,
\end{align*} 
so we may assume that $d_f(v)=0$, and the right hand side of~\eqref{eq:rbt} becomes
\begin{align}
((x_1-0.033)+\tfrac{0.033}{4})0.033+x_2x_3+x_2x_4+x_3x_4+f+x_0.\label{eq:rbt3}
\end{align}

Now~\eqref{eq:rbt3} is maximized when $x_2=x_3=x_4$ if we fix all the other variables. Note that this choice will not conflict with any other bounds. So we may assume that $x_2=x_3=x_4$.

This gives us 
\[
\tfrac12 \RBT(v) \le (x_1-0.033+\tfrac{0.033}{4})0.033
+3x_2^2
+f+x_0,   
\]
while from~\eqref{main3}:
\[
6x_1x_2-\tfrac83x_2^2- \tfrac{26}{9}x_1^2+0.027x_0-2f>0.02760856.
\]

The resulting program we want to solve is

\[
(P)
 \begin{cases}
  \text{maximize} & (x_1-0.033+\tfrac{0.033}{4})0.033+3x_2^2+f+x_0 \\
  \text{subject to}    &  0.02760856 < 6x_1x_2 - \tfrac{8}{3}x_2^2-\tfrac{26}{9}x_1^2+0.027x_0-2f,\\
                       & x_1+3x_2+x_0 = 1, \\
                       &  x_1 \geq 0, \\
                       &  x_2 \geq 0, \\
                       & x_0 \geq 0,\\
                       &  f \geq 0.\\
 \end{cases}
\]

This program can be solved using Lagrange multipliers. We give the computation in Appendix~\ref{a1}.
The optimal solution is $x_1\approx 0.246648$, $x_2\approx0.249389$, $f = 0$, and the value is
less than $0.1991$.
%
%
%
\end{proof}

\newcommand{\blackV}{0.075}

Let us call a vertex $v\in X_i$ {\em blue} if $x_i-b_i(v)\le \blackV$, and similarly {\em red} or {\em green}, and finally {\em black} if it has none of the other colors. Note that each vertex 
has exactly one of the four colors.

\begin{claim}\label{c2}
If $v\in X_1$ is black, then $d_f(v)=0$.
\end{claim}
\begin{proof}

Let $vw$ be a funky edge, and suppose that $w$ is chosen such that $d_f(w)$ is minimized over all funky neighbors of $v$. 
Therefore, $d_f(v)\times d_f(w)\le 2f$. 
  By symmetry, we may assume that $w\in X_2$ and $vw$ is red. As $G$ has maximal rainbow triangle density, recoloring $vw$ to green (making it not funky) can only reduce the number of rainbow triangles. So let us bound the number of rainbow triangles containing $vw$ before and after the recoloring.
\begin{align*}
\mbox{Before: }~\RBT(vw) 
&\le d_f(v)+d_f(w)+b_1(v)+b_2(w)+x_0; \\ 
\mbox{After: }~\RBT(vw)&\ge x_3+x_4-d_f(v,X_3\cup X_4)-d_f(w,X_3\cup X_4)
                      \ge x_3+x_4-d_f(v)-d_f(w).
\end{align*}
By the assumption that $\RBT(uw)$ does not increase  when the color of $uw$ is changed, we obtain that
\begin{align}\label{nofunky}
-b_1(v) \le b_2(w) -x_3-x_4+x_0+2d_f(v)+2d_f(w).
\end{align}
By Claim~\ref{c1}, $b_2(w) \leq x_2-0.033$, which together with $v$ being black gives
\begin{align*}
\blackV \leq x_1-b_1(v) & \le x_1+x_2-0.033-x_3-x_4+x_0+2d_f(v)+2d_f(w)\\
                        & \le 2(x_1+x_2+x_0)-0.033-1+2d_f(v)+2d_f(w).
\end{align*}
Let us maximize the right hand side using \eqref{xmin},~\eqref{fmax} and \eqref{maxfunky}.
\[
(P)
 \begin{cases}
  \text{maximize} &  2(x_1+x_2+x_0)-0.033-1+2d_f(v)+2d_f(w) \\
  \text{subject to} 
& d_f(v)\times d_f(w)\le 2f \le 2\times\funky,\\
& d_f(v)\le 1-\tfrac{35}9x_1,\\
& d_f(w)\le 1-\tfrac{35}9x_2,\\
& \Xmin \le x_1 \leq \Xmax,\\
& \Xmin \le x_2 \leq \Xmax. 
  \end{cases}
\]


In order to simplify the computation and writeup, we omit the $o(1)$ term that is coming from constraints given by \eqref{maxfunky}. The only change is that the objective functions in
the following programs contain $+o(1)$.

 To break the symmetry of $(P)$ we assume that $x_1 \leq x_2$, making the bound on $d_f(w)$ lower than the bound on $d_f(v)$.
This is allowed as all the relations of $(P)$ are symmetric in $x_1$ and $x_2$.
If  $x_0,x_1,x_2$ are fixed, the maximum of $(P)$ is attained when $d_f(v)$ is maximized, i.e.,  for $d_f(v) = 1-\tfrac{35}{9}x_1$, and then  $d_f(w)$ is maximized, i.e., for $d_f(w)=\min\{1-\tfrac{35}{9}x_2,2(\funky/d_f(v)\}$.
       
 It follows from \eqref{sumxxzero} that $x_2+x_0  < 0.2563$, which gives the following relaxation $(P_1)$ of $(P)$ with only one variable:
\[
(P_1)
 \begin{cases}
  \text{maximize} &  2(x_1+0.2563)-0.033-1+2(1-\tfrac{35}9x_1)+4(\funky/(1-\tfrac{35}9x_1)) \\
  \text{subject to} 
& \Xmin \le x_1 \leq \Xmax.
  \end{cases}
\]
Simplification of the objective function in $(P_1)$ gives $(P_1')$
\[
(P_1')
 \begin{cases}
  \text{maximize} &  1.4796-\tfrac{52}9x_1+ 0.003045924/(9-35x_1) \\    
  \text{subject to} 
& \Xmin \le x_1 \leq \Xmax.
  \end{cases}
\]
The maximum of $P_1'$ is  when $x_1 = \Xmin$ and gives $0.075 > x_1-b_1(v)$ which
contradicts $x_1-b_1(v) \geq 0.075$.

\end{proof}

\begin{claim}\label{c3}
If $v \in X_1\cup \cdots \cup X_4$ is a vertex of color $c$ that is not black,
then $v$ is not incident to any funky edges colored $c$ or to funky edges whose
non-funky color would be $c$.  
For example, a blue vertex $v\in X_1$ can be incident only to funky edges that are not blue and have the other endpoint in $X_2$ or $X_3$,
in other words,  $b_2(v)+b_3(v)+g_4(v)+r_4(v)=0$.
\end{claim}
\begin{proof}
We assume without loss of generality that $v\in X_1$ is blue.
Suppose for contradiction that  there is a vertex $w$ such that $vw$ is funky
and either $w \in X_4$ or if $w \in X_2\cup X_3$ then $uw$ is blue. 
Let us only look at the case that $w\in X_2$ and $vw$ blue, the other cases are similar.

By similar arguments as in Claim~\ref{c2} we count the number
of rainbow triangles containing $uw$ and the number after recoloring
$uw$ to green. We obtain
\begin{align*}
\mbox{Before: }~\RBT(vw) 
&\le d_f(v)+d_f(w)+r_1(v)+r_2(w)+x_0;\\
\mbox{After: }~\RBT(vw)&\ge x_3+x_4-d_f(v,X_3\cup X_4)-d_f(w,X_3\cup X_4) 
                      \ge x_3+x_4-d_f(v)-d_f(w).
\end{align*}
Since switching $vw$ to green may not increase the number of \RBT,
we get an analogue of \eqref{nofunky}
\begin{align}
-r_1(v) \leq d_f(v)+d_f(w)+r_2(w)+x_0 - (x_3+x_4-d_f(v)-d_f(w)). \label{eqr1}
\end{align}
Since $v$ is blue, $r_1(v) \leq \blackV$.
With  \eqref{maxfunky} and by adding $x_1+r_1(v)$ to both sides of  \eqref{eqr1} we get
\begin{align*}
x_1(v)&\le_{\eqref{eqr1}}  x_1+r_1(v)+x_2-x_3-x_4+x_0+2d_f(v)+2d_f(w)\\
&\le_{\eqref{maxfunky}}r_1(v)+4-\tfrac{61}9(x_1+x_2)-x_3-x_4+x_0\\
& = r_1(v) + 4 - \tfrac{52}9(x_1+x_2) - (x_0+x_1+x_2+x_3+x_4) + 2x_0\\
&\le r_1(v)+3-\tfrac{52}9\XXmin+2\cdot\Xzeromax\\
&<0.162 + r_1(v) \leq  0.237,
\end{align*}
which contradicts \eqref{xmin}.
\end{proof}


For every $v \in V(G)$ we define $d_{mono}(v):=max\{r(v),g(v),b(v)\}$.

\begin{claim}\label{c4}
The number of vertices $v$ with $d_f(v)>0$ is less than $0.00937n$. This implies that $d_f(v)<0.00937$ for all vertices in $V\setminus X_0$.
\end{claim}
\begin{proof}
Using \eqref{TCT} and the definition of $d_{mono}$ we get
\begin{align*}
0.33343492&>\tfrac13\TCT + \MONOT=\frac{1}{n}\sum_{v\in V} (r(v)^2+g(v)^2+b(v)^2) - o(1) \nonumber \\
&\ge\frac{1}{n}  \sum_{v\in V} (d_{mono}(v)^2+\tfrac12 (1-d_{mono}(v))^2) - o(1)
 \ge \frac13 - o(1),  
\end{align*}
and hence 
\begin{equation}
0.333435> \frac{1}{n}  \sum_{v\in V} (d_{mono}(v)^2+\tfrac12 (1-d_{mono}(v))^2).
\label{dsum} 
\end{equation}

By Claim~\ref{c2}, any $v$ with $d_f(v)>0$ is not black.
Without loss of generality we assume $v\in X_1$ is blue, hence $r_4(v) = g_4(v) = 0$ by Claim~\ref{c3}.
Then we have
\[
d_{mono}(v) \ge b(v)\ge x_1-\blackV+x_4 >_{\eqref{xxmin}} 0.4184. 
\]
 So
\[
d_{mono}(v)^2+\tfrac12 (1-d_{mono}(v))^2> 0.344188.  
\]
By this and \eqref{dsum}, we conclude that the number of vertices $v$ with $d_f(v)>0$ can be at most
\[
\frac{0.333435 -\tfrac13}{0.344188-\tfrac13} n
< 0.009367 <0.00937n.
\]
\end{proof}

\begin{claim}\label{c5}
For all $v \in X_1 \cup X_2 \cup X_3 \cup X_4$ we have $d_f(v)=0$.
\end{claim}
\begin{proof}
Suppose that $vw$ is funky, say $v\in X_1$, $w\in X_2$, and $vw$ is red. Then, using~\eqref{nofunky} and the bounds for $d_f(v)$ from Claim~\ref{c4},
\begin{align*}
x_1-b_1(v)+x_2-b_2(w)&\le x_1+x_2-x_3-x_4+x_0+2d_f(v)+2d_f(w)\\
&\le_{\eqref{xxmin}} \XXmax-\XXmin+\Xzeromax+4\times 0.00937 =0.0566345,  
\end{align*}
contradicting Claim~\ref{c1}, which implies that $x_1-b_1(v)+x_2-b_2(w) \geq 0.066$.
\end{proof}

%

Next, we want to show that $X_0=\emptyset$. For this, suppose that there exists $x\in X_0$. We will add $x$ to one of the $X_i$ such that $d_f(x)$ is minimal. 
By symmetry, we may assume that $x$ is added to $X_1$. Note that adding a single vertex to $X_1$ changes the density bounds we used above by at most $o(1)$.

\begin{claim}\label{c01}
For every $x \in X_0$, if $x$ was part of $X_1$ then $d_f(x)\geq0.0099$.
\end{claim}
\begin{proof}
Let $xw$ be a funky edge, where $w \in X_2$. Since $G$ is extremal, making $xw$ not funky
cannot increase the number of rainbow triangles which gives a relation analogous to $\eqref{nofunky}$.
\begin{align*}
\mbox{Before: }~\RBT(xw) 
&\le d_f(x)+b_1(x)+b_2(w)+x_0; \\ 
\mbox{After: }~\RBT(xw)&\ge x_3+x_4-d_f(x).
\end{align*}
By the assumption that $\RBT(xw)$ does not increase  when the color of $xw$ is changed, we obtain that
\begin{align}
-b_1(x) -b_2(w) \le -x_3-x_4+x_0+2d_f(x).\label{eqb1b2}
\end{align}
We also use the trivial bounds $b_1(x) \leq x_1$ and $b_2(w) \leq x_2-0.033$. Then
\pgfmathsetmacro{\dfxLB}{(2.0*\XXmin - 1.0 +0.033)}
\begin{align*}
-x_1-(x_2-0.033)&\le-b_1(x)-b_2(w)
 \le_{\eqref{eqb1b2}} -x_3-x_4+x_0+2d_f(x),
\end{align*}
\begin{align*}
2d_f(x)  &\geq x_3+x_4 +0.033-(x_0+x_1+x_2)
=  x_3+x_4 +0.033-(1-x_3-x_4)\\
 &= 2x_3+2x_4 -0.967
>_{\eqref{xxmin}}  0.019802 
> 2\times 0.0099.
\end{align*}
\end{proof}

Using yet a different way of bounding $d_f(x)$ and combining it with Claim~\ref{c01} we get the following improved bound on $d_f(x)$.

\begin{claim}\label{c02}
For every $x \in X_0$, if $x$ was part of $X_1$, then $d_f(x)>0.12866$.
\end{claim}
\begin{proof}
Suppose for a contradiction that $d_f(x)<0.12866$.
First we derive lower bounds on $d_{mono}$ of vertices in funky edges containing $x$.
Suppose that $xw$ is funky, say $w\in X_2$ and $xw$ is red. By arguments very similar to the proof of Claim~\ref{c2}, we have
\begin{align*}
\mbox{Before: }~\RBT(xw)&\le b_1(x)+b_2(w)+g_3(x)+x_0;\\
\mbox{After: }~\RBT(xw)&\ge x_3+x_4-(d_f(x)-r_2(x)).
\end{align*}
We conclude that
\begin{align*}
b_2(w)&\ge x_3+x_4-x_0-b_1(x)-d_f(x)-g_3(x)+r_2(x).
\end{align*}
Next, we give a lower bound on $d_{mono}(w)$:
\begin{align*}
d_{mono}(w)\ge b(w) = b_2(w) + x_3&\ge 2x_3+x_4-x_0-b_1(x)-d_f(x)-g_3(x)+r_2(x)\nonumber \\
&>_{\eqref{weird2}}\weirdtwomin+x_1-b_1(x)-d_f(x)-g_3(x)+r_2(x)\nonumber \\
&\ge \weirdtwomin-d_f(x)-g_3(x)+r_2(x). 
\end{align*}
Similar bounds hold for all other funky edges incident to $x$.
We give only a conclusion here:
\begin{align}
d_{mono}(w)\ge \begin{cases}
\weirdtwomin-d_f(x)-g_3(x)+r_2(x) & \text{ if } w\in X_2\text{ and } xw \text{ is red}; \\
\weirdtwomin-d_f(x)-g_4(x)+b_2(x) & \text{ if } w\in X_2\text{ and } xw \text{ is blue}; \\
\weirdtwomin-d_f(x)-r_2(x)+g_3(x) & \text{ if } w\in X_3\text{ and } xw \text{ is green}; \\
\weirdtwomin-d_f(x)-r_4(x)+b_3(x) & \text{ if } w\in X_3\text{ and } xw \text{ is blue}; \\
\weirdtwomin-d_f(x)-b_3(x)+r_4(x) & \text{ if } w\in X_4\text{ and } xw \text{ is red}; \\
\weirdtwomin-d_f(x)-b_2(x)+g_4(x) & \text{ if } w\in X_4\text{ and } xw \text{ is green.} \\
\end{cases}
\label{dmono}
\end{align}
Observe that the bound when $w\in X_2$ and $xw$ is red contains the same variables as if $w\in X_3$ and $xw$ is green. The same is true also for $w\in X_2$ with blue $xw$ and $w\in X_4$ with green $xw$ 
and also for the last pair. In order to fit the following computation on one page, we write
it only for the first pair. For the other two pairs, we use analogous operations.
It follows from \eqref{dsum}, \eqref{dmono} and $d_f(x) = r_2(x)+g_3(x)+b_2(x)+g_4(x)+b_3(x)+r_4(x)$ that
\begin{align*}
0.333435>&_{\eqref{dsum}}~\frac{1}{n}\sum_{v\in V(G)} (d_{mono}(v)^2+\tfrac12 (1-d_{mono}(v))^2)\\
\ge&~\tfrac13(1-d_f(x))+\tfrac12d_f(x)\\
&~+r_2(x)[\tfrac32(\weirdtwomin-d_f(x)-g_3(x)+r_2(x))^2-(\weirdtwomin-d_f(x)-g_3(x)+r_2(x))]\\
&~+g_3(x)[\tfrac32(\weirdtwomin-d_f(x)+g_3(x)-r_2(x))^2-(\weirdtwomin-d_f(x)+g_3(x)-r_2(x))]\\
&~+b_2(x)(...)+g_4(x)(...)
+b_3(x)(...)+r_4(x)(...)\\
=&~\tfrac13(1-d_f(x))+\tfrac12d_f(x)\\
&~+d_f(x)( \tfrac32(\weirdtwomin-d_f(x))^2-\weirdtwomin+d_f(x))\\
&~+r_2(x)[3(\weirdtwomin-d_f(x))(r_2(x)-g_3(x))+\tfrac32(r_2(x)-g_3(x))^2-(r_2(x)-g_3(x)ß)]\\
&~+g_3(x)[3(\weirdtwomin-d_f(x))(g_3(x)-r_2(x))+\tfrac32(g_3(x)-r_2(x))^2-(g_3(x)-r_2(x))]+ \cdots \\
=&~\tfrac13(1-d_f(x))+\tfrac12d_f(x)+d_f(x)( \tfrac32(\weirdtwomin-d_f(x))^2-\weirdtwomin+d_f(x))\\
&~+r_2(x)[3(\weirdtwomin-d_f(x))-1)(r_2(x)-g_3(x))+\tfrac32(r_2(x)-g_3(x))^2]\\
&~+g_3(x)[3(\weirdtwomin-d_f(x))-1)(g_3(x)-r_2(x))+\tfrac32(g_3(x)-r_2(x))^2]+ \cdots \\
=&~\tfrac13(1-d_f(x))+\tfrac12d_f(x)+d_f(x)( \tfrac32(\weirdtwomin-d_f(x))^2-\weirdtwomin+d_f(x))\\
&~+(3(\weirdtwomin-d_f(x))-1)(r_2(x)-g_3(x))^2+\tfrac32(r_2(x)-g_3(x))^2(r_2(x)+g_3(x))
+ \cdots.
\end{align*}
If $d_f(x) < 0.12866$, then $3(\weirdtwomin-d_f(x))-1 > 0$. Hence,
\[
(3(\weirdtwomin-d_f(x))-1)(r_2(x)-g_3(x))^2+\tfrac32(r_2(x)-g_3(x))^2(r_2(x)+g_3(x)) \geq 0,
\]
and we can obtain the following lower bound:
\begin{align*}
0.333435\ge&~ \tfrac13(1-d_f(x))+\tfrac12d_f(x)+d_f(x)(\tfrac32(\weirdtwomin-d_f(x))^2-\weirdtwomin+d_f(x))\\
=&~ \tfrac32d_f(x)^3+(1-3\times\weirdtwomin)d_f(x)^2+(\tfrac16+\tfrac32\times\weirdtwomin^2-\weirdtwomin)d_f(x)+\tfrac13,
\end{align*}
so
\[
0\ge\tfrac32d_f(x)^3 -0.454961 d_f(x)^2 +0.03449825 d_f(x)-0.000102.
\]
All $d_f(x)$ that satisfy the last inequality are in $(-\infty,0.0031) \cup (0.12866,0.1716)$.
Claim~\ref{c01} implies that $d_f(x)$ is not in $(-\infty,0.0031)$, hence $d_f(x) > 0.12866$, which is a contradiction to the assumption $d_f(x) < 0.12866$.
\end{proof}

\begin{claim}\label{c03}
The set $X_0$ is empty.
\end{claim}
\begin{proof}
We will show that $\RBT(x)<0.397$ for any $x\in X_0$, contradicting Corollary~\ref{triD}. For the ease of notation, we will write $r_i$ for $r_i(x)$ etc.
\begin{align*}
\tfrac{1}{2}\RBT(x)\le~&
~ \tfrac{1}{2}x_0^2+x_0(1-x_0)\\
&+r_1g_1+r_1b_1+g_1b_1
+r_2g_2+r_2b_2+g_2b_2
+r_3g_3+r_3b_3+g_3b_3
+r_4g_4+r_4b_4+g_4b_4\\
&+r_1(b_2+g_4)+b_2g_4+g_1(b_3+r_4)+b_3r_4
+b_1(r_2+g_3)+r_2g_3+g_2(r_3+b_4) +r_3b_4\\
\le_{(*)}&~\tfrac{1}{2}x_0^2+x_0(1-x_0)+0.1945(1-x_0)^2 
<_{\eqref{xzero}} 0.1982,
\end{align*}
where $(*)$ comes from a massive computation described in Appendix~\ref{a03}. This contradiction proves the claim.
%
%
\end{proof}


\begin{claim}\label{c8}
For $n$ large enough, we have $|X_i|-|X_j|\le 1$.
\end{claim}
\newcommand{\avg}{\mathrm{avg}}
\begin{proof}
By symmetry, for a contradiction we assume $|X_1|-|X_2|\ge 2$. Then we move a vertex from $X_1$ to $X_2$ and show that doing so increases the number of rainbow triangles.
Recall that $\RBT(v)$ denotes the rooted density of $\RBT$ at $v$.
Denote 
\begin{align*}
F_{\avg}(m) &= \frac{1}{m}\sum_{v \in V(G_m)}\RBT(v)\binom{m-1}{2} = 3\frac{F(m)}{m},
\end{align*}
where $G_m$ is an extremal graph on $m$ vertices. Let
\[
l = \lim_{m\rightarrow \infty}\frac{F_{\avg}(m)}{\binom{m-1}{2}}.
\]
The limit exists since $F_{\avg}(m)/\binom{m-1}{2} = F(m)/\binom{m}{3}$ is non-increasing and lower bounded by 0.4.
Corollary~\ref{triD} implies that $0.40005 \geq l \geq 0.4$. 
Let $a_i = |X_i|=nx_i$ for $i \in \{1,2,3,4\}$.
We delete $v$ from $X_1$, where $\RBT(v)$ is minimized  over vertices in $X_1$, and add a duplicate $w'$ of $w \in X_2$,
where $\RBT(w)$ is maximized over vertices in $X_2$. We color $ww'$ arbitrarily. 
\begin{align*}
\text{Before: } \RBT(v)\binom{n-1}{2} &\leq F_{\avg}(a_1) + a_2a_3 + a_2a_4 + a_3a_4, \\
\text{After: } \RBT(w')\binom{n-1}{2} &\geq F_{\avg}(a_2) + (a_1-1)a_3 + (a_1-1)a_4 + a_3a_4.
\end{align*}
Since $G$ is extremal, $\RBT(v) \geq \RBT(w')$.
Now we estimate $F_{\avg}(a_1) - F_{\avg}(a_2)$. Since $F_{\avg}(m)/\binom{m-1}{2}$ is non-increasing and its limit is $l$, for $n$ large enough we have
\[
F_{\avg}(a_1) = a_1^2 \cdot l/2 + \eps_1a_1^2, \qquad F_{\avg}(a_2) = a_2^2 \cdot l/2 + \eps_2a_2^2
\]
and $\eps_1\le \eps_2\le 0.01$.
Then we have $F_{\avg}(a_1) - F_{\avg}(a_2) \leq (l/2+0.01)(a_1^2 - a_2^2)$ and
obtain
\begin{align*}
0 &\leq (\RBT(v) - \RBT(w'))\binom{n-1}{2} \\&\leq F_{\avg}(a_1) + a_2a_3 + a_2a_4 + a_3a_4 - F_{\avg}(a_2) - (a_1-1)a_3 - (a_1-1)a_4 - a_3a_4 \\
  &\leq  (0.5l+0.01)(a_1^2 - a_2^2) - (a_1-1-a_2)(a_3+a_4)
     < 0.22(a_1- a_2)(a_1+a_2) - 0.5(a_1-a_2)(a_3+a_4)\\
  &\leq (a_1- a_2)(0.22(a_1+a_2)-0.5(a_3+a_4))< 0,
\end{align*}
which is a contradiction.
\end{proof}


Claim~\ref{c8} gives a proof of Theorem~\ref{thmrecurs}.
\begin{proof}[Proof of Theorem~\ref{thmrecurs}]
Let $G$ be an extremal graph on $n$ vertices, where $n$ is sufficiently large, such that Claim~\ref{c8} holds.
Denote $a=|X_1|$, $b=|X_2|$, $c=|X_3|$ and $d=|X_4|$.
By Claim~\ref{c8},  $a,b,c,d$ are as equal as possible. Moreover, by Claims~\ref{c5} and \ref{c03}, rainbow
triangles are either entirely in one $X_i$ for $1 \leq i \leq 4$, or intersect three of the $X_i$'s.
It then follows from the extremality of $G$ that
\[
F(n)=F(a)+F(b)+F(c)+F(d)+abc+abd+acd+bcd,
\]
which completes the proof of the recurrence.
Notice that $X_1$, $X_2$, $X_3$, and $X_4$ satisfy the claimed blow-up property
by Claim~\ref{c5}.
\end{proof}

\section{Extremal graphs}
Now that we know the limit object, we look at the extremal graphs on
$n$ vertices. Using a standard blow-up argument, Theorem~\ref{mainthm} implies
that any $3$-edge-colored graph $G$ contains at most $(n^3-n)/15$ rainbow triangles.
\begin{corollary}
Every $3$-edge-colored graph on $n$ vertices contains at most $(n^3-n)/15$ rainbow triangles.
\end{corollary}
\begin{proof}
Suppose there exists a $3$-edge-colored graph $G$ on $k$ vertices with $r=(k^3-k)/15 + \ell$ rainbow triangles for some $\ell > 0$.
Without loss of generality, $G$ is a $3$-edge-coloring of $K_n$.
Let $G_0 := G$ and $G_{i+1}$, for $i\in\Nat$, will be obtained by blowing up every vertex of $G$ by a factor $k^i$ and placing
$G_i$ inside every blob.
It follows that $v(G_i) = k^{i+1}$ and $F(G_i) = k^{3i} \cdot r + k\cdot F(G_{i-1})$.
Recall that $F(G_i)$ denotes the number of rainbow triangles in $G_i$.
Expanding the recurrence, it follows that
\[
F(G_i) = \sum_{j=0}^i k^{3j} \cdot k^{i-j} \cdot r 
= \frac{k^{3i}\left(k^3-k + 15\ell\right)}{15} \cdot\sum_{t=0}^i \frac{1}{k^{2t}}
.
\]
Therefore,
\[
\lim_{i\to\infty} \frac{F(G_i)}{{v(G_i)\choose3}} =
\frac{k^2}{k^2-1}\cdot \frac{6\cdot\left(k^{3i+3}-k^{3i+1} + 15\ell \cdot k^{3i}\right)}{15 \cdot k^{3(i+1)}}
= \frac{2}{5} \cdot \left(1+\frac{15\ell}{k^3-k}\right).
\]
However, any convergent subsequence of $(G_i)_{i\in\Nat}$ converges to a homomorphism with the
density of rainbow triangles equal to $\frac{2}{5} \cdot \left(1+\frac{15\ell}{k^3-k}\right)>\frac{2}{5}$,
which contradicts Theorem~\ref{mainthm}.
\end{proof}
The iterated blow-up of $\RB1111$ shows that for $n$ being a power of $4$, the
bound $(n^3-n)/15$ is best possible. In this case, we
show that the iterated blow-up of $\RB1111$ is actually a unique extremal construction.

\begin{proof}[Proof of Theorem~\ref{thm4k}.] 
Denote by $R^\ell$ the $(\ell-1)$-times iterated blow-up of $\RB1111$, so $R^\ell$ has $4^\ell$ vertices.
Theorem~\ref{thm4k} is easily seen to be true for $k=1$, so
suppose for a contradiction that there is a graph $G$ on $n=4^k$ vertices with $F(G) = F(n) = (n^3-n)/15$ that
is not isomorphic to $R^k$ for a minimal $k\ge 2$.

If $G$ has the structure described in Theorem~\ref{thmrecurs}, then $G$ is isomorphic to $R^k$ by the minimality of $k$, a contradiction.
Therefore, $V(G)$ cannot be partitioned into four parts $X_1,X_2,X_3,X_4$ with $|X_i|=4^{k-1}$ as described in Theorem~\ref{thmrecurs}.

Fix an integer $\ell$ such that $4^\ell > n_0$, where $n_0$ is taken from the statement of Theorem~\ref{thmrecurs}. 
Let $\overline{G}$ be the graph obtained by blowing up every vertex of $G$ by a factor of $4^\ell$, 
and inserting $R^\ell$ in every part. It follows that $\overline{G}$ has $4^{k+\ell}$ vertices, and
\[
F(\overline{G}) = n\cdot F\left(R^\ell\right) + F(G) \cdot 4^{3\ell}
=\frac{n\cdot4^{3\ell}-n\cdot4^{\ell} + n^3 \cdot 4^{3\ell} - n\cdot4^{3\ell}}{15}
=\frac{4^{3(k+\ell)}-4^{k+\ell}}{15}
.
\]
So $\overline{G}$ must be extremal. However, Theorem~\ref{thmrecurs} implies that $\overline{G}$
can be partitioned into four parts $\overline{X}_1,\overline{X}_2,\overline{X}_3,\overline{X}_4$ with $|\overline{X}_i|=4^{k+\ell-1}$ as described in Theorem~\ref{thmrecurs}.
Since any two vertices from $V(\overline{G})$ that arise from blowing up the same vertex of $G$ need to be in the same part,
the partition $\overline{X}_1,\overline{X}_2,\overline{X}_3,\overline{X}_4$ provides also a partition of the vertices of $G$.
But this is a partition of $G$ into four parts of the same size as described in Theorem~\ref{thmrecurs}, a contradiction.
\end{proof}

\section*{Acknowledgement}
We are grateful to Hong Liu for fruitful discussions at the beginning of the project.


%
%

\bibliographystyle{abbrv}
\bibliography{refs.bib}

\appendix

\section{Giving bounds on the $x_i$}\label{ax}
Here we show how to prove \eqref{xzero} -- \eqref{sumxxzero}.
Suppose we want to derive the upper bound from \eqref{xmin}. It means solving the following program:
\[
(P)
 \begin{cases}
  \text{maximize} & x_1 \\
  \text{subject to}  &  2\sum_{1\le i<j\le 4}x_ix_j-2f-\tfrac{26}{9}\sum_{1\le i\le 4}x_i^2+\tfrac{27}{1000}x_0 >0.02760856,\\
                       & x_1+x_2+x_3+x_4+x_0 = 1, \\
                       &  x_i \geq 0 \text{ for } i \in \{0,\ldots,4\},\\
                       &  f \geq 0.\\
 \end{cases}
\]
As a quick check, it can be written to a heuristic  online solver APMonitor. We provide
the source of the program in file \texttt{APM.xi.txt}. However, 
this method
may get stuck in local optima, so it does not provide a proof of global maximization.

A rigorous way is to use the method of Lagrange Multipliers. Since we need to solve
several of the programs, we implemented the method in \href{http://www.sagemath.org}{Sage}. 
We provide a commented code in file \texttt{solve-xi.py}.

\section{The computation in Claim~\ref{c1}}\label{a1}

Recall that we want to solve the following program
\[
(P)
 \begin{cases}
  \text{maximize} & (x_1-0.033+\tfrac{0.033}{4})0.033+3x_2^2+f+x_0 \\
  \text{subject to}    &  0.02760856 < 6x_1x_2 - \tfrac{8}{3}x_2^2-\tfrac{26}{9}x_1^2+0.027x_0-2f,\\
                       & x_1+3x_2+x_0 = 1, \\
                       &  x_1 \geq 0, \\
                       &  x_2 \geq 0, \\
                       & x_0 \geq 0,\\
                       &  f \geq 0.\\
 \end{cases}
\]
We give a solution using Lagrange multipliers. 
We also implemented a script in \href{http://www.sagemath.org}{Sage} performing the computation.
The script is in file \texttt{solve-claim12.py}.

First observe that if $x_1=0$ or $x_2=0$, then the program is not feasible. Hence
$x_1 > 0$ and $x _2 > 0$. We are left with inequalities  $x_0 \geq 0$ and $f \geq 0$,
which may be tight. Moreover, we always use $x_1+3x_2+x_0 = 1$ for substitution.
 To solve this, we divide the analysis in four cases, and use Lagrange multipliers again:

\noindent\textbf{Case 1:}  
   If $f=0$ and $x_0=0$, this comes down to solving
  \[
  (P)
   \begin{cases}
    \text{maximize} & 0.033x_1+\tfrac13(1-x_1)^2 -\tfrac{3}{4}(0.033)^2 \\
    \text{subject to}    &  0.02760856 <2x_1(1-x_1)-\tfrac8{27}(1-x_1)^2- \tfrac{26}{9}x_1^2.
   \end{cases}
  \]
  The constraint can be simplified to $0.02760856< -\tfrac2{27} (4 - 35 x_1 + 70 x_1^2)$.
  %
 This quadratic program in one variable has the optimal solution $x_1\approx 0.24424$, and so $\tfrac12 \RBT^v<0.1985$. 
  
\noindent\textbf{Case 2:}   
  If $f=0$ and $x_0>0$, it comes down to solving
 \[
 (P)
  \begin{cases}
   \text{maximize} & 0.033x_1+3x_2^2-x_1-3x_2+1-\tfrac{3}{4}(0.033)^2 \\
   \text{subject to}    &  0.02760856<6x_1x_2-\frac83x_2^2- \frac{26}{9}x_1^2+0.027(1-x_1-3x_2),\\
                        & 0.24 \le x_1 \le 0.26,\\
                        & 0.24 \le x_2 \le 0.26. \\
   \end{cases}
 \]
   
 Taking gradients, we get
 \[
 \begin{pmatrix}
 -0.967\\ -3+6x_2
 \end{pmatrix}
 =\lambda 
 \begin{pmatrix}
 -\tfrac{52}9 x_1+6x_2-0.027\\ 6x_1-\tfrac{16}3 x_2-0.081
 \end{pmatrix},
 \]
 which gives $x_1\approx 0.24662$, $x_2\approx 0.24936$, 
 and
 $\tfrac12 \RBT^v<0.19991$ as the only feasible solution.

 \noindent\textbf{Case 3:}  
 If $f>0$ and $x_0=0$, it comes down to solving
 \[
 (P)
  \begin{cases}
   \text{maximize} &  0.033x_1+\tfrac13(1-x_1)^2+f-\tfrac{3}{4}(0.033)^2 \\
   \text{subject to}    &  0.02760856 <2x_1(1-x_1)-\tfrac8{27}(1-x_1)^2- \tfrac{26}{9}x_1^2-2f.
   \end{cases}
 \]
 %
 %
 %
 The constraint can be simplified to $ 0.02760856<-\tfrac2{27} (4 - 35 x_1+ 70 x_1^2)-2f$.
 Taking gradients, we get
 \[
 \begin{pmatrix}
 0.033-\tfrac23+2x_1\\ 1
 \end{pmatrix}
 =\lambda 
 \begin{pmatrix}
 \tfrac{70}{27}-\tfrac{280}{27}x_1\\ -2
 \end{pmatrix},
 \]
 whose solution $x_1 \approx 0.20803$ together with the constraint implies $f<0$, a contradiction.
 
\noindent\textbf{Case 4:}  
 If $f>0$ and $x_0>0$, it comes down to solving
\[
 (P)
  \begin{cases}
\text{maximize} & 0.033x_1 +3x_2^2 +1-x_1-3x_2+f\\
\text{subject to} & 0.02760856<6x_1x_2-\frac83x_2^2- \frac{26}{9}x_1^2+0.027(1-x_1-3x_2)-2f.
 \end{cases}
 \]
 Taking gradients, we get
 \[
 \begin{pmatrix}
 0.967\\ -3+6x_2\\1
 \end{pmatrix}
 =\lambda 
 \begin{pmatrix}
 -\tfrac{52}9 x_1+6x_2-0.027\\ 6x_1-\tfrac{16}3 x_2-0.081\\-2
 \end{pmatrix}.
 \]
 Similarly to the previous case,  
 we again have $f<0$, a contradiction.

\section{The computation in Claim~\ref{c03}}\label{a03}
The term we want to maximize does not include anything from $X_0$, so we can
assume that $x_0 = 0$. Since $r_1+g_1+b_1=x_1$, we can use bounds involving $x_1,\ldots,x_4$.
First, we use $x_1+x_2+x_3+x_4=1$.
Then we use the lower bounds for~\eqref{xmin} on all $x_i$.
We also use the four bounds implied by Claim~\ref{c02} (since there are four options where to put $x$). 
Finally, we add 
the bounds $r_i,g_i,b_i\ge 0$.
So we solve the following program:
\[
(P) = \begin{cases}
\text{maximize}  
&r_1g_1+r_1b_1+g_1b_1
+r_2g_2+r_2b_2+g_2b_2\\
&+r_3g_3+r_3b_3+g_3b_3
+r_4g_4+r_4b_4+g_4b_4\\
&+r_1(b_2+g_4)+b_2g_4+g_1(b_3+r_4)+b_3r_4\\
&+b_1(r_2+g_3)+r_2g_3+g_2(r_3+b_4) +r_3b_4\\
\text{subject to}
& \sum_{i=1}^4r_i+g_i+b_i = 1,\\
& r_i+g_i+b_i \geq\Xmin  \text{ for } i \in \{1,2,3,4\},\\
& r_2+b_2+g_3+b_3+r_4+g_4 \geq 0.12866,\\
& r_1+b_1+r_3+g_3+g_4+b_4 \geq 0.12866,\\
& g_1+b_1+r_2+g_2+r_4+b_4 \geq 0.12866,\\
& r_1+g_1+g_2+b_2+r_3+b_3 \geq 0.12866,\\
& r_i,g_i,b_i \geq 0   \text{ for } i \in \{1,2,3,4\}.\\
\end{cases}
\]

The optimal solution to the program has value less than 0.1945 and it is
achieved at
$r_1\approx  0.03854,
g_1\approx   0.16720,
b_1\approx   0.03854,
r_2=0,
g_2\approx 0.24670,
b_2=0,
r_3\approx 0.19243,
g_3=0,
b_3\approx 0.06658,
r_4\approx  0.06208,
g_4= 0,
b_4\approx 0.18792
$.




For each of the bounds, we consider the two cases that the bound is active (i.e.~tight) or inactive, giving us a total of $2^{20}$ cases. In each of the cases, we have to solve a system of linear equations with up to $12$ variables, and check the solution for feasibility. Obviously, this is done by a computer using rational arithmetic.
We wrote a program in Sage which performs the computation. We reduce the number of programs to solve by eliminating
the cases where some sets of constraints cannot be tight at the same time. For example, it
is not possible that $r_1=g_1=b_1=0$ at the same time.
Note that feasible solutions with dimension greater than zero will occur again as lower dimensional solutions in cases with more active bounds, so we only have to analyze discrete solutions.  We could use symmetries, and we could analyze the feasibility polytope closer to only check the faces which actually appear (the program Polymake~\cite{polymake} can yield this output), reducing the number of cases to check to a few thousand. But we decided to use this brute-force analysis, as this makes it easier to check the code, and the running time is still very reasonable.

The code performing the computation as well as the outputs can be downloaded at \oururl.

\end{document}